\documentclass[10pt,english]{article}
\pdfoutput=1

\usepackage[latin1]{inputenc}
\usepackage[T1]{fontenc}
\usepackage{lmodern}
\usepackage{babel}
\usepackage{amsmath,amssymb,amsthm}
\usepackage[marginratio={1:1,1:1},totalwidth=400pt,totalheight=600pt]{geometry}
\usepackage{microtype}
\usepackage{hyperref}
\usepackage{url}
\usepackage{tikz-cd}

\newtheorem{theorem}{Theorem}[section]
\newtheorem{lemma}[theorem]{Lemma}
\newtheorem{proposition}[theorem]{Proposition}
\newtheorem{corollary}[theorem]{Corollary}
\newtheorem{lems}{Lemma}[subsection]

\theoremstyle{definition}

\newtheorem{example}[theorem]{Example}
\newtheorem{remark}[theorem]{Remark}

\makeatletter
\def\journal#1{\gdef\@journal{\footnotetext{Published in #1.}}}
\let\@journal\@empty
\def\subjclass#1{\gdef\@subjclass{\footnotetext{2010 \emph{Mathematics Subject Classification}: #1.}}}
\let\@subjclass\@empty
\def\keywords#1{\gdef\@keywords{\footnotetext{\emph{Key words and phrases}: #1.}}}
\let\@keywords\@empty

\def\@maketitle{
\newpage
\null
\vskip 2em
\begin{center}
\let \footnote \thanks \@journal \@subjclass \@keywords
{\LARGE \@title \par}
\vskip 1.5em
{\large
\lineskip .5em
\begin{tabular}[t]{c}
\@author
\end{tabular}\par}
\vskip 1em
{\large \@date}
\end{center}
\par
\vskip 1.5em}
\makeatother

\title{$C^*$-algebras generated by projective representations \\ of free nilpotent groups}
\author{Tron {\AA}nen Omland}
\date{Updated version, July 2016}
\journal{J.\ Operator Theory, 73(1):3-25, 2015}
\subjclass{Primary 46L05; Secondary 20C25, 20F18, 20J06, 22D25}
\keywords{free nilpotent group, projective unitary representation, twisted group $C^*$-algebra, simplicity, two-cocycle, group cohomology, Heisenberg group, noncommutative $n$-torus}

\begin{document}

\maketitle

\begin{abstract}
We compute the two-cocycles (or multipliers) of the free nilpotent groups of class $2$ and rank $n$ and give conditions for simplicity of the corresponding twisted group $C^*$-algebras. These groups are representation groups for $\mathbb{Z}^n$ and can be considered as a family of generalized Heisenberg groups with higher-dimensional center. Their group $C^*$-algebras are in a natural way isomorphic to continuous fields over $\mathbb{T}^{\frac{1}{2}n(n-1)}$ with the noncommutative $n$-tori as fibers. In this way, the twisted group $C^*$-algebras associated with the free nilpotent groups of class $2$ and rank $n$ may be thought of as ``second order'' noncommutative $n$-tori.
\end{abstract}

\section*{Introduction}
\addcontentsline{toc}{section}{Introduction}

The discrete Heisenberg group may be described as the group generated by three elements $u_1$, $u_2$, and $v_{12}$ satisfying the commutation relations
\begin{displaymath}
[u_1,v_{12}]=[u_2,v_{12}]=1 \quad\text{ and }\quad [u_1,u_2]=v_{12}.
\end{displaymath}
The group has received much attention in the literature,
partly because it is one of the easiest examples of a nonabelian torsion-free group.
Moreover,
the continuous Heisenberg group (see below) is a connected nilpotent Lie group that arises in certain quantum mechanical systems.

As a natural consequence of this attention, several classes of generalized Heisenberg groups have been investigated.
For example, in \cite{Milnes-Walters-4,Milnes-Walters-5} Milnes and Walters describe the four and five-dimensional nilpotent groups,
and in \cite{Lee-Packer-N,Lee-Packer-C} Lee and Packer study the finitely generated torsion-free two-step nilpotent groups with one-dimensional center.

In this paper, on the other hand, we will consider a family of generalized Heisenberg groups, denoted by $G(n)$ for $n\geq 2$, with larger center.
The groups $G(n)$ are the so-called free nilpotent groups of class $2$ and rank $n$ and will be defined properly in Section~\ref{G(n)}.
Here we also provide further motivation for our investigation of these groups.
Inspired by the work of Packer \cite{Packer-H} we compute the second cohomology group $H^2(G(n),\mathbb{T})$ of $G(n)$
and study the structure of the twisted group $C^*$-algebras $C^*(G(n),\sigma)$ associated with two-cocycles $\sigma$ of $G(n)$.

Section~\ref{two-cocycles} is devoted to two-cocycle calculations,
where we decompose $G(n)$ into a semidirect product and apply techniques introduced by Mackey \cite{Mackey}.
In particular, we will see that
\begin{displaymath}
H^2(G(n),\mathbb{T})\cong\mathbb{T}^{\frac{1}{3}(n+1)n(n-1)},
\end{displaymath}
and in Theorem~\ref{main-theorem} we give explicit formulas for the two-cocycles of $G(n)$ up to similarity.

Next, in Section~\ref{algebras} we describe $C^*(G(n),\sigma)$ as a universal $C^*$-algebra of a set of generators and relations.
Then we construct the algebra
that in a natural way appear as a continuous field over the compact space $H^2(G(n),\mathbb{T})$ with $C^*(G(n),\sigma)$ as fibers.
We also explain that for $n=2$, this algebra is the group $C^*$-algebra of the free nilpotent group of class $3$ and rank $2$.

In Section~\ref{simplicity} we investigate the center of $C^*(G(n),\sigma)$
and give conditions for simplicity of these twisted group $C^*$-algebras in Theorem~\ref{simplicity-theorem} and Corollary~\ref{simplicity-formula}.

Finally, in Section~\ref{isomorphisms} we study the automorphism group of $G(n)$
and discuss isomorphism invariants of $C^*(G(n),\sigma)$ coming from $\operatorname{Aut}G(n)$.

\subsection*{Acknowledgements}
\addcontentsline{toc}{subsection}{Acknowledgements}

This version is from July 2016 and includes several computations and details that were left out in the published version of the article,
in addition to a reformulation of Proposition~\ref{similarity} in terms of general semidirect products,
and a new Remark~\ref{new} that makes a connection to more recent work.

A comprehensive summary of this paper was part of the author's PhD dissertation \cite{phd-thesis}
at the Norwegian University of Science and Technology (NTNU) submitted in May~2013.

The author would like to thank Erik B{\'e}dos, first for suggesting the problem of computing the two-cocycles of the group $G(3)$,
and then for providing valuable comments throughout this work.
The author would also like to thank the referee of Journal of Operator Theory for several valuable comments and suggestions.

This research was partially supported by the Research Council of Norway (NFR).

\section{The free nilpotent groups \texorpdfstring{$G(n)$}{G(n)} of class \texorpdfstring{$2$}{2} and rank \texorpdfstring{$n$}{n}}\label{G(n)}

For each natural number $n\geq 2$, let $G(n)$ be the group generated by elements $\{u_i\}_{1\leq i\leq n}$ and $\{v_{jk}\}_{1\leq j<k\leq n}$ subject to the relations
\begin{equation}\label{G(n)-relations}
[v_{jk},v_{lm}]=[u_i,v_{jk}]=1 \quad\text{ and }\quad [u_j,u_k]=v_{jk}
\end{equation}
for $1\leq i\leq n$, $1\leq j<k\leq n$, and $1\leq l<m\leq n$.
Clearly, $G(2)$ is the usual discrete Heisenberg group.
For some purposes, it can be useful to set $G(1)=\langle u_1\rangle\cong\mathbb{Z}$.
Note that $G(n)$ is generated by $n+\frac{1}{2}n(n-1)=\frac{1}{2}n(n+1)$ elements.

The group $G(n)$ is called the free nilpotent group of class $2$ and rank $n$.
Indeed, $G(n)$ is a free object on $n$ generators in the category of nilpotent groups of step at most two.
To see this, note first that $G(n)$ is the group generated by $\{u_i\}_{i=1}^n$ subject to the relations
that all commutators of order greater than two involving the generators are trivial.
Let $G'(n)$ be any other nilpotent group of step at most two and let $\{u_i'\}_{i=1}^n$ be any set of $n$ elements in $G'(n)$.
Then there is a unique homomorphism from $G(n)$ to $G'(n)$ that maps $u_i$ to $u'_i$ for $1\leq i\leq n$.
Of course, every free object on $n$ generators in this category is isomorphic to $G(n)$.
For a more extensive treatment of free nilpotent groups, see the article on Terence Tao's website \cite{Tao} (see also 2.\ in the list below).

Furthermore, we will need the following concrete realization, say $\widetilde{G}(n)$, of $G(n)$.
For $n\geq 2$, we denote the elements of $\widetilde{G}(n)$ by
\begin{displaymath}
r=(r_1,\dotsc,r_n,r_{12},r_{13},\dotsc,r_{n-1,n})\footnote{To be absolutely precise,
the entries with double index are colexicographically ordered, that is, $(i,j)<(k,l)$ if $j<l$ or if $j=l$ and $i<k$.},
\end{displaymath}
where all entries are integers, and define multiplication by
\begin{displaymath}
\begin{split}
r\cdot s=(r_1+s_1,\dotsc,r_n+s_n&,\\ r_{12}+s_{12}+r_1s_2&,r_{13}+s_{13}+r_1s_3,\dotsc,r_{n-1,n}+s_{n-1,n}+r_{n-1}s_n).
\end{split}
\end{displaymath}

By letting $u_i$ have $1$ in the $i$'th spot and $0$ else and $v_{jk}$ have $1$ in the $jk$'th spot and $0$ else,
the relations \eqref{G(n)-relations} are satisfied for these elements.
Next, we define the map
\begin{displaymath}
\widetilde{G}(n)\longrightarrow G(n),\quad r\longmapsto v_{12}^{r_{12}}\cdots v_{n-1,n}^{r_{n-1,n}} \cdot u_n^{r_n}\cdots u_1^{r_1},
\end{displaymath}
and then it is not difficult to see that $\widetilde{G}(n)$ is isomorphic to $G(n)$.
Henceforth, we will not distinguish between $G(n)$ and the realization $\widetilde{G}(n)$ just described, but this should cause no confusion.

Denote by $V(n)$ the subgroup of $G(n)$ generated by the $v_{jk}$'s.
Then $V(n)$ coincides with the center $Z(G(n))$ of $G(n)$ and
\begin{displaymath}
V(n)=Z(G(n))\cong\mathbb{Z}^{\frac{1}{2}n(n-1)}.
\end{displaymath}
Indeed, both this and the next observations follow after noticing that
\begin{displaymath}
r\cdot s\cdot r^{-1}=(s_1,\dotsc,s_n,s_{12}+r_1s_2-s_1r_2,\dotsc,s_{n-1,n}+r_{n-1}s_n-s_{n-1}r_n).
\end{displaymath}
Moreover, consider the subgroups $G(n-1)$ and $H(n)$ of $G(n)$ defined by
\begin{displaymath}
\begin{split}
G(n-1) &= \langle u_i,v_{jk} : 1\leq i\leq n-1, 1\leq j<k\leq n-1 \rangle,\\
H(n) &= \langle u_n,v_{jn} : 1\leq j<n\rangle.
\end{split}
\end{displaymath}
Note that $G(n-1)$ sits inside $G(n)$ as a subgroup and that $H(n)\cong\mathbb{Z}^n$ is a normal subgroup of $G(n)$.
Clearly, we have $G(n)/V(n)\cong\mathbb{Z}^n$ and $G(n)/H(n)\cong G(n-1)$.
Therefore, there are short exact sequences
\begin{displaymath}
\begin{tikzcd}
1 \arrow{r} & V(n) \arrow{r} & G(n) \arrow{r} & \mathbb{Z}^n \arrow{r} & 1
\end{tikzcd}
\end{displaymath}
and
\begin{displaymath}
\begin{tikzcd}
1 \arrow{r} & H(n) \arrow{r} & G(n) \arrow{r} & G(n-1) \arrow{r} & 1
\end{tikzcd}
\end{displaymath}
where the second one splits and the first does not.
In particular, $G(n)$ is a central extension of $\mathbb{Z}^n$ by $\mathbb{Z}^{\frac{1}{2}n(n-1)}$ and consequently, $G(n)$ is a two-step nilpotent group.

\vspace{1em}

To motivate our investigation of $G(n)$, we present a few aspects about these groups and some appearances in the literature.

\begin{enumerate}

\item[1.]
Consider in the first place the \emph{continuous} Heisenberg group.
We will represent this group in two different ways, $G_{\text{matrix}}$ and $G_{\text{wedge}}$,
both with elements $(x,x')=(x_1,x_2,x')\in\mathbb{R}^3$, i.e.\ $x=(x_1,x_2)\in\mathbb{R}^2$, and with multiplication as follows.
For $G_{\text{matrix}}$ we define
\begin{displaymath}
(x_1,x_2,x')(y_1,y_2,y')=\left(x_1+y_1,x_2+y_2,x'+y'+x_1y_2\right),
\end{displaymath}
and for $G_{\text{wedge}}$ we set
\begin{displaymath}
(x_1,x_2,x')(y_1,y_2,y')=\left(x_1+y_1,x_2+y_2,x'+y'+\tfrac{1}{2}(x_1y_2-x_2y_1)\right).
\end{displaymath}
One can deduce that $G_{\text{matrix}}\cong G_{\text{wedge}}$.
To motivate the notation, note that $G_{\text{matrix}}$ can be represented as matrix multiplication in $M_3(\mathbb{R})$ if one identifies
\begin{displaymath}
(x_1,x_2,x')\longleftrightarrow
\begin{bmatrix}
1&x_1&x'\\
0&1&x_2\\
0&0&1
\end{bmatrix},
\end{displaymath}
and that the multiplication in $G_{\text{wedge}}$ may be written as
\begin{displaymath}
(x,x')(y,y')=\left(x+y,x'+y'+\tfrac{1}{2}(x\wedge y)\right).
\end{displaymath}

In general, the wedge product on $\mathbb{R}^n$ is defined as a certain bilinear map (see e.g.\ \cite[p.~79]{Spivak})
\begin{displaymath}
\mathbb{R}^n\times\mathbb{R}^n\to\textstyle\bigwedge\nolimits^2(\mathbb{R}^n),
\end{displaymath}
where $\bigwedge^2(\mathbb{R}^n)$ is a $\frac{1}{2}n(n-1)$-dimensional real vector space.
The elements of $\bigwedge^2(\mathbb{R}^n)$ are called bivectors and if $\{e_i\}_{i=1}^n$ is a basis for $\mathbb{R}^n$,
then $\{e_i\wedge e_j\}_{i<j}$ is a basis for $\bigwedge^2(\mathbb{R}^n)$.
For every $n\geq 2$, define the group $\widehat{G}(n,\mathbb{R})$ with elements
\begin{displaymath}
(x,x')\in\mathbb{R}^n\oplus\textstyle\bigwedge\nolimits^2(\mathbb{R}^n),\quad \text{where }x=(x_1,\dotsc, x_n),\quad x'=(x'_{12},x'_{13},\dotsc,x'_{n-1,n}),
\end{displaymath}
and where multiplication is given by
\begin{displaymath}
(x,x')(y,y')=\left(x+y,x'+y'+\tfrac{1}{2}(x\wedge y)\right).
\end{displaymath}
This group is of dimension $n+\frac{1}{2}n(n-1)=\frac{1}{2}n(n+1)$.
Remark especially that if $n=3$, the wedge product can be identified with the vector cross product on $\mathbb{R}^3$.
That is, the product in $\widehat{G}(3,\mathbb{R})$ is given by
\begin{displaymath}
(x,x')(y,y')=\left(x+y,x'+y'+\tfrac{1}{2}(x\times y)\right).
\end{displaymath}
It is not hard to see that $\widehat{G}(n,\mathbb{R})$ is isomorphic to the group consisting of the same elements,
but with multiplication given by
\begin{equation}\label{order}
(x,x')(y,y')=\left(x+y,x'+y'+(x_1y_2,x_1y_3,\dotsc,x_{n-1}y_n)\right).
\end{equation}
Let $G(n,\mathbb{R})$ denote the group defined by \eqref{order}.
Then $G(n)$ is the integer version of $G(n,\mathbb{R})$.

We also mention that Nielsen \cite{Nielsen} has classified all the six-dimensional connected, simply connected, nilpotent Lie groups.
In this setting, $G(3,\mathbb{R})$ is the group denoted by $G_{6,15}$.

\item[2.]
One may define the free nilpotent group $G(m,n)$ of class $m$ and rank $n$ for every $m\geq\nobreak 1$.
Indeed, $G(m,n)$ is the group generated by $\{u_i\}_{i=1}^n$
subject to the relations that all commutators of order greater than $m$ involving the generators are trivial.
More precisely, for $m=1,2,3$ and $n\geq 2$, we have that $G(m,n)$ can be described as the groups with presentations
\begin{equation}\label{class-3}
\begin{split}
G(1,n)&=\langle \{u_i\}_{i=1}^n : [u_i,u_j]=1\rangle\cong\mathbb{Z}^n,\\
G(2,n)&=\langle \{u_i\}_{i=1}^n : [[u_i,u_j],u_k]=1\rangle=G(n),\\
G(3,n)&=\langle \{u_i\}_{i=1}^n : [[[u_i,u_j],u_k],u_l]=1\rangle,
\end{split}
\end{equation}
and it should now be clear how to define $G(m,n)$ for all $m\geq 1$ and $n\geq 2$.
Finally, we set $G(m,1)=\langle u_1\rangle\cong\mathbb{Z}$ for each $m\geq 1$.
Moreover, for all $m,n\geq 1$, the group $G(m,n)$ is the free object on $n$ generators in the category of nilpotent groups of step at most $m$.
In particular, notice that $G(m,n)$ is $m$-step nilpotent and that
\begin{equation}\label{quotient}
G(m,n)\cong G(m+1,n)/Z(G(m+1,n)).
\end{equation}
Again, we refer to \cite{Tao} for additional details.

In \cite[Section~4]{Milnes-Walters-5}
Milnes and Walters describe the simple quotients of the $C^*$-algebra associated with a five-dimensional group denoted by $H_{5,4}$.
One can check that $H_{5,4}$ is isomorphic to the group $G(3,2)$. See Remark~\ref{conjecture} for more about this group.

\item[3.]
The group $G(3)$ is briefly discussed by Baggett and Packer \cite[Example~4.3]{Baggett-Packer}.
The purpose of that paper is to describe the primitive ideal space of group $C^*$-algebras of some two-step nilpotent groups.
However, $G(3)$ only serves as an example of a group the authors could not handle.

\item[4.]
Let $n\geq 2$.
It is well-known that the group $C^*$-algebra $A=C^*(G(n))$ may be described as
the universal $C^*$-algebra generated by unitaries $\{U_i\}_{1\leq i\leq n}$ and $\{V_{jk}\}_{1\leq j<k\leq n}$ satisfying the relations
\begin{displaymath}
[V_{jk},V_{lm}]=[U_i,V_{jk}]=I \quad\text{ and }\quad [U_j,U_k]=V_{jk}
\end{displaymath}
for all $1\leq i\leq n$, $1\leq j<k\leq n$, and $1\leq l<m\leq n$.

For $\lambda=(\lambda_{12},\lambda_{13},\dotsc,\lambda_{n-1,n})\in\mathbb{T}^{\frac{1}{2}n(n-1)}$, let $\mathcal{A}_{\lambda}$ be the noncommutative $n$-torus.
It is the universal $C^*$-algebra generated by unitaries $\{W_i\}_{i=1}^n$ and relations $[W_i,W_j]=\lambda_{ij}I$ for $1\leq i<j\leq n$.
The universal property of $A$ gives that for each $\lambda$ in $\mathbb{T}^{\frac{1}{2}n(n-1)}$ there is a surjective $^*$-homomorphism
\begin{displaymath}
\pi_{\lambda}\colon A\to\mathcal{A}_{\lambda}
\end{displaymath}
satisfying $\pi_{\lambda}(U_i)=W_i$ for $1\leq i\leq n$ and $\pi_{\lambda}(V_{jk})=\lambda_{jk}I$ for $1\leq j<k\leq n$.

Furthermore, $A$ has center $Z(A)=C^*(\{V_{jk}\}_{1\leq j<k\leq n})\cong C^*(V(n))$.
Indeed, this is the case since $G(n)$ is amenable
and its finite conjugacy classes are precisely the one-point sets of central elements (see Lemma~\ref{center} below).
Therefore, we set
\begin{displaymath}
T=\operatorname{Prim}{Z(A)}\cong\widehat{Z(A)}=\mathbb{T}^{\frac{1}{2}n(n-1)}.
\end{displaymath}
Let $\lambda$ be a primitive ideal of $Z(A)$ identified with an element of $\mathbb{T}^{\frac{1}{2}n(n-1)}$.
Let $\mathcal{I}_{\lambda}$ be the ideal of $A$ generated by $\lambda$, that is, the ideal generated by $\{V_{jk}-\lambda_{jk}I : 1\leq j<k\leq n\}$.
It is clear that $\mathcal{I}_{\lambda}\subseteq\operatorname{ker}{\pi_{\lambda}}$.
By the universal property of $\mathcal{A}_{\lambda}$, there is a $^*$-homomorphism
\begin{displaymath}
\rho\colon\mathcal{A}_{\lambda}\to A/\mathcal{I}_{\lambda}
\end{displaymath}
such that $\rho(W_i)=U_i+\mathcal{I}_{\lambda}$ for $1\leq i\leq n$.
Hence, $\rho\circ\pi_{\lambda}$ coincides with the quotient map $A\to A/\mathcal{I}_{\lambda}$
and consequently, $\operatorname{ker}{\pi_{\lambda}}\subseteq\mathcal{I}_{\lambda}$.
Therefore, $\mathcal{A}_{\lambda}\cong A/\mathcal{I}_{\lambda}$ and $\pi_{\lambda}$ may be regarded as the quotient map $A\to A/\mathcal{I}_{\lambda}$.

For an element $a$ of $A$, let $\tilde{a}$ be the section $T\to\bigsqcup_T\mathcal{A}_{\lambda}$ given by $\tilde{a}(\lambda)=\pi_{\lambda}(a)$
and let $\widetilde{A}=\{\tilde{a}\mid a\in A\}$ be the set of all such sections.
Then the following can be deduced from the Dauns-Hofmann Theorem \cite{Dauns-Hofmann}.

\begin{theorem}\label{continuous-field}
The triple $(T,\{\mathcal{A}_{\lambda}\},\widetilde{A})$ consisting of the base space $T$, $C^*$-algebras $\mathcal{A}_{\lambda}$ for each $\lambda$ in $T$,
and the set of sections $\widetilde{A}$, is a full continuous field of $C^*$-algebras.
Moreover, the $C^*$-algebra associated with this continuous field is naturally isomorphic to $A$.
\end{theorem}

This result may be obtained as a corollary to \cite[Theorem~1.2]{Packer-Raeburn} which employs tools of Williams \cite{Williams} related to Fell bundle theory,
by taking $G=G(n)$ and $\sigma=1$ in that theorem.
It is also a special case of \cite[Corollary~2.3]{Baggett-Packer-2}.
Our proof is more direct and partly inspired by \cite[Theorem~1.1]{Anderson-Paschke} which covers the case where $n=2$.

\vspace{1em}

From the above discussion it now follows that $G(n)$ is a \emph{representation group} for $\mathbb{Z}^n$ in the sense of Moore \cite{Moore-2}.
In this case, that means $G(n)$ is (up to isomorphism) the unique central extension of $\mathbb{Z}^n$ by $\widehat{H^2(\mathbb{Z}^n,\mathbb{T})}$
such that the ordinary irreducible representation theory of $G(n)$ coincides with the projective irreducible representation theory of $\mathbb{Z}^n$.

This fact plays an important role in \cite{ENO},
where the noncommutative principal torus bundles over locally compact spaces are classified up to equivariant Morita equivalence.
As explained in \cite[Section~2]{ENO}, the group $C^*$-algebra of $G(n)$ serves as a ``universal'' bundle in this classification.

We refer to \cite[Section~4]{EW} for more information on representation groups,
where the groups $G(n,\mathbb{R})$ and $G(n)$ are treated particularly in \cite[Example~4.7]{EW}.


\end{enumerate}

\section{The two-cocycles of the free nilpotent groups \texorpdfstring{$G(n)$}{G(n)}}\label{two-cocycles}

Let $G$ be \emph{any} discrete group with identity $e$.
A function $\sigma\colon G\times G\to\mathbb{T}$ satisfying
\begin{gather*}
\sigma(r,s)\sigma(rs,t)=\sigma(r,st)\sigma(s,t)\\
\sigma(r,e)=\sigma(e,r)=1
\end{gather*}
for all elements $r,s,t\in G$ is called a \emph{two-cocycle of $G$ with values in $\mathbb{T}$} (or a \emph{multiplier of $G$}).
Moreover, two two-cocycles $\sigma$ and $\tau$ are said to be \emph{similar}, written $\sigma\sim\tau$, if
\begin{displaymath}
\tau(r,s)=\beta(r)\beta(s)\overline{\beta(rs)}\sigma(r,s)
\end{displaymath}
for all $r,s\in G$ and some function $\beta\colon G\to\mathbb{T}$.
The set of similarity classes of two-cocycles of $G$ is an abelian group under pointwise multiplication.
This group is the second cohomology group $H^2(G,\mathbb{T})$.

\vspace{1em}

Let $G$ be a semidirect product of a normal subgroup $H$ and a subgroup $K$.
By properties of the semidirect product, the elements of $G$ can be uniquely written as products $ab$,
where $a$ belongs to $H$ and $b$ belongs to $K$.
Define the action $\alpha$ of $K$ on $H$ by $\alpha_b(a)=bab^{-1}$.
One often writes $G=H\rtimes_{\alpha} K$,
but to simplify the notation, we will still denote the elements of $G$ by $ab$ instead $(a,b)$
and write the group product in $G$ as $(ab)(a'b')=a\alpha_{b}(a')bb'$ for $a,a'\in H$ and $b,b'\in K$.
Hopefully, the reader is familiar with semidirect products so that this does not cause any confusion.

\vspace{1em}

Next, we apply Mackey's theorem \cite[Theorem 9.4]{Mackey} and obtain the following result.
\begin{theorem}\label{semidirect}
Every two-cocycle of $G$ is similar to a two-cocycle $\sigma$ of $G$ of the form
\begin{equation}\label{cocycle-formula}
\sigma(a'b,ab')=\sigma_{H}(a',\alpha_b(a))g(a,b)\sigma_{K}(b,b'),
\end{equation}
where $\sigma_{H}$ and $\sigma_{K}$ are two-cocycles of $H$ and $K$, respectively,
\begin{displaymath}
g\colon H\times K\to\mathbb{T}
\end{displaymath}
is a function such that $g(a,e)=g(e,b)=1$ for all $a\in H$, $b\in K$,
and $\sigma_{H}$ and $g$ satisfy
\begin{equation}\label{identities}
\begin{split}
g(aa',b)&=\sigma_{H}(\alpha_b(a),\alpha_b(a'))\overline{\sigma_{H}(a,a')}\cdot g(a,b)g(a',b),\\
g(a,bb')&=g(\alpha_{b'}(a),b)g(a,b').
\end{split}
\end{equation}
Moreover, for every choice of $\sigma_{H}$, $g$, and $\sigma_{K}$ satisfying the conditions above, $\sigma$ is a two-cocycle of $G$.
\end{theorem}

\begin{proposition}\label{similarity}
Let $(\sigma_{H},g,\sigma_{K})$ and $(\sigma_{H}',g',\sigma_{K}')$ be triples satisfying the conditions of Theorem~\ref{semidirect} and
let $\sigma$ and $\sigma'$ be the corresponding two-cocycles of $G$.
Then $\sigma\sim\sigma'$ if and only if the following conditions hold:
\begin{enumerate}
\item[(i)] $\sigma_{K}\sim\sigma_{K}'$,
\item[(ii)] There exists a function $\beta\colon H\to\mathbb{T}$ such that
\begin{displaymath}
\begin{split}
\sigma_{H}'(a,a')&=\overline{\beta(a)\beta(a')}\beta(aa')\sigma_{H}(a,a'),\\
g'(a,b)&=\beta(\alpha_b(a))\overline{\beta(a)}g(a,b).
\end{split}
\end{displaymath}
\end{enumerate}
\end{proposition}

\begin{remark}
If (ii) holds, then $\sigma_{H}\sim\sigma_{H}'$.
If $\sigma_{H}\sim\sigma_{H}'$ and $\beta$ and $\beta'$ are two functions implementing the similarity,
then $\beta'=f\cdot\beta$ for some homomorphism $f\colon H\to\mathbb{T}$.
\end{remark}

\begin{proof}[Proof of Proposition~\ref{similarity}]
Suppose $\sigma\sim\sigma'$.
Then there exists some $\gamma\colon G\to\mathbb{T}$ such that
\begin{equation}\label{gamma}
\sigma(a'b,ab')=\gamma(a'b)\gamma(ab')\overline{\gamma(a'bab')}\sigma'(a'b,ab')
\end{equation}
for all $a,a'\in H$ and $b,b'\in K$.
In particular, if $a=a'=e$, then
\begin{displaymath}
\sigma_{K}(b,b')=\gamma(b)\gamma(b')\overline{\gamma(bb')}\sigma_{K}'(b,b')
\end{displaymath}
for all $b,b'\in K$, so $\sigma_{K}\sim\sigma_{K}'$.
Moreover, the formula \eqref{cocycle-formula} from Theorem~\ref{semidirect} with $a=e$ and $b=e$ gives that
\begin{displaymath}
\sigma(a',b')=1=\sigma'(a',b')
\end{displaymath}
for all $a'\in H$ and $b'\in K$.
Applying this fact to \eqref{gamma} shows that $\gamma(a'b')=\gamma(a')\gamma(b')$ for all $a'\in H$ and $b'\in K$.
Define $\beta$ on $H$ by $\beta(a)=\gamma(a)$.
Then, by letting $b=b'=e$ in \eqref{cocycle-formula} and \eqref{gamma}, we get
\begin{displaymath}
\sigma_{H}'(a',a)=\overline{\beta(a')\beta(a)}\beta(a'a)\sigma_{H}(a',a)
\end{displaymath}
for all $a',a\in H$.
Furthermore, by letting $a'=e$ and $b'=e$ in \eqref{cocycle-formula} and \eqref{gamma}, we compute
\begin{displaymath}
\begin{split}
g(a,b)&=\gamma(b)\gamma(a)\overline{\gamma(ba)}g'(a,b)\\
&=\gamma(b)\gamma(a)\overline{\gamma(\alpha_b(a)b)}g'(a,b)\\
&=\gamma(b)\gamma(a)\overline{\gamma(\alpha_b(a))}\overline{\gamma(b)}g'(a,b)\\
&=\gamma(a)\overline{\gamma(\alpha_b(a))}g'(a,b)\\
&=\beta(a)\overline{\beta(\alpha_b(a))}g'(a,b)
\end{split}
\end{displaymath}
for all $a\in H$ and $b\in K$.

Assume next that $\beta$ is such that (ii) holds, and that (i) holds through $\delta$, that is,
\begin{displaymath}
\sigma_{K}(b,b')=\delta(b)\delta(b')\overline{\delta(bb')}\sigma_{K}'(b,b').
\end{displaymath}
Define $\gamma$ on $G$ by $\gamma(ab)=\beta(a)\delta(b)$.
Then
\begin{displaymath}
\begin{split}
\sigma(a'b,ab')&=\sigma_{H}(a',\alpha_b(a))g(a,b)\sigma_{K}(b,b')\\
&=\beta(a')\beta(\alpha_b(a))\overline{\beta(a'\alpha_b(a))}\sigma_{H}'(a',\alpha_b(a))\\
&\quad\cdot\beta(a)\overline{\beta(\alpha_b(a))}g'(a,b)\cdot\delta(b)\delta(b')\overline{\delta(bb')}\sigma_{K}'(b,b')\\
&=\beta(a')\delta(b)\cdot\beta(a)\delta(b')\cdot\overline{\beta(a'\alpha_b(a))\delta(bb')}\sigma'(a'b,ab')\\
&=\gamma(a'b)\gamma(ab')\overline{\gamma(a'bab')}\sigma'(a'b,ab').
\end{split}
\end{displaymath}
\end{proof}

Fix $n\geq 2$.
To compute the two-cocycles of $G(n)$ up to similarity, we will proceed in the following way.
Consider $G(n)$ as the split extension of $G(n-1)$ by $H(n)$ as described in Section~\ref{G(n)}.
We will identify the elements
\begin{displaymath}
\begin{split}
a&=(0,\dotsc,0,a_n,0,\dotsc,0,a_{1n},\dotsc,a_{n-1,n}),\\
b&=(b_1,\dotsc,b_{n-1},0,b_{12},\dotsc,b_{n-2,n-1},0,\dotsc,0),
\end{split}
\end{displaymath}
of $H(n)$ and $G(n-1)$, respectively, with ones of the form
\begin{displaymath}
\begin{split}
a&\longleftrightarrow (a_n,a_{1n},\dotsc,a_{n-1,n}),\\
b&\longleftrightarrow (b_1,\dotsc,b_{n-1},b_{12},\dotsc,b_{n-2,n-1}).
\end{split}
\end{displaymath}
The elements of $G(n)$ will be written as products $ab$,
where $a$ belongs to $H(n)$ and $b$ belongs to $G(n-1)$,
and the action $\alpha$ of $G(n-1)$ on $H(n)$ is then given by
\begin{displaymath}
\alpha_b(a)=bab^{-1}=(a_n,a_{1n}+b_1a_n,\dotsc,a_{n-1,n}+b_{n-1}a_n).
\end{displaymath}

\begin{remark}
In the published version, Theorem~\ref{semidirect} and Proposition~\ref{similarity} were only shown to hold for $G(n)$,
not for any semidirect product.

Proposition~\ref{similarity} can be deduced from \cite[Appendix~2]{Packer-Raeburn},
but in any case it may be useful to give a proof by a direct computation.
\end{remark}

Let $\tau_n$ be a two-cocycle of $G(n)$ coming from a pair $(\sigma_{H(n)},g_n)$, that is,
\begin{equation}\label{tilde-formula}
\tau_n(a'b,ab')=\sigma_{H(n)}(a',\alpha_b(a))g_n(a,b),
\end{equation}
where $(\sigma_{H(n)},g_n)$ satisfies \eqref{identities}.
By Theorem~\ref{semidirect} and Proposition~\ref{similarity}, every two-cocycle of $G(n)$ that is trivial on $G(n-1)$ is similar to one of this form.
Denote the abelian group of similarity classes of two-cocycles of this type by $\widetilde{H}^2(G(n),\mathbb{T})$.

\begin{corollary}\label{inductively}
The second cohomology group of $G(n)$ may be decomposed as
\begin{displaymath}
H^2(G(n),\mathbb{T})=\widetilde{H}^2(G(n),\mathbb{T})\oplus H^2(G(n-1),\mathbb{T})=\bigoplus_{k=2}^n\widetilde{H}^2(G(k),\mathbb{T}).
\end{displaymath}
\end{corollary}

\begin{proof}
It follows from Theorem~\ref{semidirect} and Proposition~\ref{similarity} (see our comment above) that
\begin{displaymath}
H^2(G(n),\mathbb{T})=\widetilde{H}^2(G(n),\mathbb{T})\oplus H^2(G(n-1),\mathbb{T}).
\end{displaymath}
Thus, the second inequality is proven by induction after noticing that
\begin{displaymath}
\{1\}=H^2(\mathbb{Z},\mathbb{T})=H^2(G(1),\mathbb{T})=\widetilde{H}^2(G(1),\mathbb{T}).
\qedhere
\end{displaymath}
\end{proof}

\begin{theorem}\label{main-theorem}
We have
\begin{displaymath}
H^2(G(n),\mathbb{T})\cong\mathbb{T}^{\frac{1}{3}(n+1)n(n-1)},
\end{displaymath}
and for each set of
$\frac{1}{3}(n+1)n(n-1)$ parameters
\begin{displaymath}
\{\lambda_{i,jk} : 1\leq i\leq k, 1\leq j<k\leq n\}\subseteq\mathbb{T},
\end{displaymath}
the associated $[\sigma]$ in $H^2(G(n),\mathbb{T})$ may be represented by
\begin{equation}\label{main-formula}
\begin{split}
\sigma(r,s)&=\prod_{i<j<k}\lambda_{i,jk}^{s_{jk}r_i+s_kr_{ij}}\lambda_{j,ik}^{s_{ik}r_j+s_k(r_ir_j-r_{ij})}\\
&\quad\cdot\prod_{j<k}\lambda_{j,jk}^{s_{jk}r_j+\frac{1}{2}s_kr_j(r_j-1)}\lambda_{k,jk}^{r_k(s_{jk}+r_js_k)+\frac{1}{2}r_js_k(s_k-1)}.
\end{split}
\end{equation}
\end{theorem}

The proof of this theorem will be given in Section~\ref{main-proof}.

See the paragraph following Theorem~\ref{universal} for an explanation of why $\lambda_{i,jk}$ for $i>k$ is not involved in the above.

\begin{example}
For $G(1)\cong\mathbb{Z}$ there are no nontrivial two-cocycles.
The two-cocycles of the usual Heisenberg group $G(2)$ are, up to similarity, given by two parameters (as computed in \cite[Proposition~1.1]{Packer-H}):
\begin{equation}\label{G(2)-formula}
\sigma(r,s)=\lambda_{1,12}^{s_{12}r_1+\frac{1}{2}s_2r_1(r_1-1)}\lambda_{2,12}^{r_2(s_{12}+r_1s_2)+\frac{1}{2}r_1s_2(s_2-1)}
\end{equation}
The two-cocycles of $G(3)$ are, up to similarity, given by eight parameters:
\begin{displaymath}
\begin{split}
\sigma(r,s)&=\lambda_{1,23}^{s_{23}r_1+s_3r_{12}}\lambda_{2,13}^{s_{13}r_2+s_3(r_1r_2-r_{12})}\\
&\quad\cdot\lambda_{1,12}^{s_{12}r_1+\frac{1}{2}s_2r_1(r_1-1)}\lambda_{2,12}^{r_2(s_{12}+r_1s_2)+\frac{1}{2}r_1s_2(s_2-1)}\\
&\quad\cdot\lambda_{1,13}^{s_{13}r_1+\frac{1}{2}s_3r_1(r_1-1)}\lambda_{3,13}^{r_3(s_{13}+r_1s_3)+\frac{1}{2}r_1s_3(s_3-1)}\\
&\quad\cdot\lambda_{2,23}^{s_{23}r_2+\frac{1}{2}s_3r_2(r_2-1)}\lambda_{3,23}^{r_3(s_{23}+r_2s_3)+\frac{1}{2}r_2s_3(s_3-1)}
\end{split}
\end{displaymath}
\end{example}

\begin{remark}\label{homology}
One may associate a Lyndon-Hochschild-Serre spectral sequence with the extension (see e.g.\ \cite[6.8.2]{Weibel}):
\begin{displaymath}
\begin{tikzcd}
1 \arrow{r} & V(n) \arrow{r} & G(n) \arrow{r} & \mathbb{Z}^n \arrow{r} & 1
\end{tikzcd}
\end{displaymath}
By applying \cite[Theorem 4]{Kuzmin-Semenov} to this sequence, one can compute the second homology group of $G(n)$
(which is recently also done more generally for $G(m,n)$ in \cite[Proposition~2.1]{Szymik}),
and deduce that
\begin{displaymath}
H_2(G(n),\mathbb{Z})\cong\mathbb{Z}^{\frac{1}{3}(n+1)n(n-1)},
\end{displaymath}
which gives that $H^2(G(n),\mathbb{T})\cong\mathbb{T}^{\frac{1}{3}(n+1)n(n-1)}$ after dualizing,
using the universal coefficient theorem for cohomology.
However, this does not give an explicit description of $H^2(G(n),\mathbb{T})$.
\end{remark}

\subsection{Proof of \texorpdfstring{Theorem~\ref{main-theorem}}{Theorem~2.6}}\label{main-proof}

We will in this proof first compute $\widetilde{H}^2(G(n),\mathbb{T})$ through several lemmas
and then use Corollary~\ref{inductively} to conclude the argument.

\begin{lems}\label{lemma-Hn}
Every element of $\widetilde{H}^2(G(n),\mathbb{T})$ may be represented by a pair $(\sigma_{H(n)},g_n)$,
where $\sigma_{H(n)}$ is a two-cocycle of $H(n)$ given by
\begin{equation}\label{sigma-Hn}
\sigma_{H(n)}(a',a)=\prod_{i=1}^{n-1}\lambda_i^{a'_na_{in}}
\end{equation}
for some $\lambda_1,\dotsc,\lambda_{n-1}\in\mathbb{T}$, and $g_n$ satisfies
\begin{equation}\label{g-Hn}
g_n(a+a',b)=\Big(\prod_{i=1}^{n-1}\lambda_i^{b_ia_na'_n}\Big)g_n(a,b)g_n(a',b)
\end{equation}
for all $a,a'\in H(n)$ and $b\in G(n-1)$.
\end{lems}

\begin{proof}
Every element of $\widetilde{H}^2(G(n),\mathbb{T})$ may be represented by a two-cocycle of the form \eqref{tilde-formula},
that is, by a pair $(\sigma_{H(n)},g_n)$ satisfying \eqref{identities}.

Moreover, it is well-known (see e.g.\ \cite{Backhouse}) that every two-cocycle of $H(n)\cong\mathbb{Z}^n$ is similar to one of the form
\begin{displaymath}
\sigma_{H(n)}(a',a)=\prod_{1\leq i\leq n-1}\lambda_i^{a'_na_{in}}\cdot\prod_{1\leq j<k\leq n-1}\mu_{jk}^{a'_{jn}a_{kn}}
\end{displaymath}
for some sets of scalars $\{\lambda_i\}_{1\leq i\leq n-1},\{\mu_{jk}\}_{1\leq j<k\leq n-1}\subseteq\mathbb{T}$.
Since $H(n)$ is abelian, \eqref{identities} gives that
\begin{displaymath}
\begin{split}
\sigma_{H(n)}(\alpha_b(a),\alpha_b(a'))\overline{\sigma_{H(n)}(a,a')}&=g_n(a+a',b)\overline{g_n(a,b)g_n(a',b)}\\
&=g_n(a'+a,b)\overline{g_n(a',b)g_n(a,b)}\\
&=\sigma_{H(n)}(\alpha_b(a'),\alpha_b(a))\overline{\sigma_{H(n)}(a',a)}
\end{split}
\end{displaymath}
for all $a,a'\in H(n)$ and $b\in G(n-1)$.
Furthermore, we have
\begin{displaymath}
\begin{split}
\sigma_H(\alpha_b(a)&,\alpha_b(a'))\overline{\sigma_H(a,a')}\\
=&\prod_{1\leq i\leq n-1}\lambda_i^{a_n(a'_{in}+b_ia'_n)-a_na'_{in}}\cdot\prod_{1\leq j<k\leq n-1}\mu_{jk}^{(a_{jn}+b_ja_n)(a'_{kn}+b_ka'_n)-a_{jn}a'_{kn}}\\
=&\prod_{1\leq i\leq n-1}\lambda_i^{b_ia_na'_n}\cdot\prod_{1\leq j<k\leq n-1}\mu_{jk}^{b_ja'_{kn}a_n+b_ka_{jn}a'_n+b_jb_ka_na'_n}.
\end{split}
\end{displaymath}
This is equal to $\sigma_H(\alpha_b(a'),\alpha_b(a))\overline{\sigma_H(a',a)}$ for all $a,a'\in H(n)$ and $b\in G(n-1)$ if and only if
the expression remains unchanged under the substitution $a\longleftrightarrow a'$, that is, if and only if all the $\mu_{jk}$'s are $1$.
\end{proof}

\begin{lems}
For every element of $\widetilde{H}^2(G(n),\mathbb{T})$ there is a \emph{unique} associated pair $(\sigma_{H(n)},g_n)$
satisfying the conditions of Lemma~\ref{lemma-Hn} such that
\begin{equation}\label{g-trivial}
g_n(u_n,u_i)=1 \quad\text{for all}\quad 1\leq i\leq n-1.
\end{equation}
\end{lems}

\begin{proof}
Suppose that $(\sigma_{H(n)},g_n)$ satisfies \eqref{sigma-Hn} and \eqref{g-Hn}.
Let $f\colon H(n)\to\mathbb{T}$ be the homomorphism determined by $f(u_n)=1$ and $f(v_{in})=\overline{g_n(u_n,u_i)}$ for all $1\leq i\leq n-1$
and define $g_n'$ by $g_n'(a,b)=f(\alpha_b(a))\overline{f(a)}g_n(a,b)$.
Then, $g_n'(u_n,u_i)=1$ for all $1\leq i\leq n-1$ and by Proposition~\ref{similarity},
$(\sigma_{H(n)},g_n')$ determines a two-cocycle of $H(n)$ in the same similarity class as the one coming from $(\sigma_{H(n)},g_n)$.

Suppose now that there are two pairs $(\sigma_{H(n)},g_n)$ and $(\sigma_{H(n)}',g_n')$ both satisfying the conditions of Lemma~\ref{lemma-Hn}.
Then $\sigma_{H(n)}'=\sigma_{H(n)}$, so by Proposition~\ref{similarity} and the succeeding remark,
there is a homomorphism $f\colon H(n)\to\mathbb{T}$ such that
\begin{displaymath}
g_n'(a,b)=f(\alpha_b(a))\overline{f(a)}g_n(a,b)=\Big(\prod_{i=1}^{n-1}f(v_{in})^{a_nb_i}\Big)g_n(a,b)
\end{displaymath}
for all $a\in H(n)$ and $b\in G(n-1)$.
In particular,
\begin{displaymath}
g_n'(u_n,u_i)=f(v_{in})g_n(u_n,u_i)\quad\text{for all}\quad 1\leq i\leq n-1,
\end{displaymath}
so that $g_n'=g_n$ if $g_n'(u_n,u_i)=g_n(u_n,u_i)$ for all $1\leq i\leq n-1$.
\end{proof}

In the forthcoming lemmas we fix an element of $\widetilde{H}^2(G(n),\mathbb{T})$, and let $(\sigma_{H(n)},g)$ be the unique associated pair
satisfying \eqref{sigma-Hn}, \eqref{g-Hn}, and \eqref{g-trivial} for some set of scalars $\{\lambda_i\}_{i=1}^{n-1}\subseteq\mathbb{T}$.

For computational reasons, we introduce the following notation.
For $a=(a_n,a_{1n},\dotsc,a_{n-1,n})$ in $H(n)$, we write $a=w(a)+z(a)$, where $w(a)=(a_n,0,\dotsc,0)$,
and $z(a)$ is the ``central part'', i.e.\ $z(a)=(0,a_{1n},\dotsc,a_{n-1,n})$.
Similarly, for $b=(b_1,\dotsc,b_{n-1},b_{12},\dotsc,b_{n-2,n-1})$ in $G(n-1)$, we write $b=w(b)z(b)$,
where $w(b)=(b_1,\dotsc,b_{n-1},0,\dotsc,0)$ and $z(b)=(0,\dotsc,0,b_{12},\dotsc,b_{n-2,n-1})$.
Note that $\alpha_b(a)=a$ if either $w(a)$ or $w(b)$ is trivial, i.e.\ if either $a$ or $b$ is central.

\begin{lems}\label{decomposition}
For all $a\in H(n)$ and $b\in G(n-1)$ we have
\begin{displaymath}
g(a,b)=g(w(a),w(b))g(w(a),z(b))g(z(a),w(b)).
\end{displaymath}
\end{lems}

\begin{proof}
It follows immediately from Lemma~\ref{lemma-Hn} that if $a,a'\in H(n)$ and $w(a)$ or $w(a')$ is $0$, then
\begin{equation}\label{w(a)-trivial}
g(a+a',b)=g(a,b)g(a',b),
\end{equation}
hence,
\begin{displaymath}
g(a,b)=g(w(a)+z(a),b)=g(w(a),b)g(z(a),b)
\end{displaymath}
for all $a\in H(n)$ and $b\in G(n-1)$.
If $b'\in G(n-1)$ and $w(b')=e$, then $b'$ is central and $\alpha_{b'}(a)=a$ for all $a\in H(n)$.
Therefore,
\begin{equation}\label{central-b}
g(a,b)g(a,b')=g(a,bb')=g(a,b'b)=g(\alpha_{b}(a),b')g(a,b)
\end{equation}
for all $a\in H(n)$, $b\in G(n-1)$.
By \eqref{w(a)-trivial}, we then get
\begin{displaymath}
\begin{split}
1&=g(\alpha_{b}(a),b')\overline{g(a,b')}\\
&=g(a+(0,b_1a_n,\dotsc,b_{n-1}a_n),b')\overline{g(a,b')}\\
&=g(a,b')g((0,b_1a_n,\dotsc,b_{n-1}a_n),b')\overline{g(a,b')}\\
&=g((0,b_1a_n,\dotsc,b_{n-1}a_n),b')
\end{split}
\end{displaymath}
for all $a\in H(n)$ and $b\in G(n-1)$.
Consequently, since this holds for all $a\in H(n)$ and $b\in G(n-1)$, and central $b'\in G(n-1)$,
we get that if $\tilde{a}$ and $\tilde{b}$ are \emph{any} elements in $H(n)$ and $G(n-1)$, respectively, then $g(z(\tilde{a}),z(\tilde{b}))=1$.
Moreover, \eqref{central-b} also imply that if $b,b'\in G(n-1)$ and \emph{either} $w(b)$ \emph{or} $w(b')$ is equal to $e$,
that is, either $b$ or $b'$ is central, then
\begin{equation}\label{w(b)-trivial}
g(a,bb')=g(a,b)g(a,b').
\end{equation}
Hence, by \eqref{w(b)-trivial} and \eqref{w(a)-trivial},
\begin{displaymath}
\begin{split}
g(a,b)&=g(a,w(b)z(b))=g(a,w(b))g(a,z(b))\\
      &=g(w(a),w(b))g(z(a),w(b))g(w(a),z(b))\cdot 1
\end{split}
\end{displaymath}
for all $a\in H(n)$ and $b\in G(n-1)$.
\end{proof}

\begin{lems}\label{z-part}
For all $a\in H(n)$ and $b,b'\in G(n-1)$ we have
\begin{displaymath}
\begin{split}
g(z(a),w(b))&=\prod_{i,j=1}^{n-1} g(v_{in},u_j)^{a_{in}b_j},\\
g(w(a),z(b))&=\prod_{1\leq i<j\leq n} g(u_n,v_{ij})^{a_nb_{ij}}=\prod_{1\leq i<j\leq n}\Big(\overline{g(v_{in},u_j)}g(v_{jn},u_i)\Big)^{a_nb_{ij}},
\end{split}
\end{displaymath}
and
\begin{equation}\label{b-homomorphism}
g(a,bb')=\Big(\prod_{i,j=1}^{n-1}g(v_{in},u_j)^{b_i'b_ja_n}\Big)g(a,b)g(a,b').
\end{equation}
\end{lems}

\begin{proof}
Let $z(H(n))=\{z(a) \mid a\in H(n)\}$ and $z(G(n-1))=\{z(b) \mid b\in G(n-1)\}$.
Then $g$ is a bihomomorphism when restricted to $z(H(n))\times G(n-1)$ or $H(n)\times z(G(n-1))$.
Therefore, the first two identities hold.
Indeed, this follows directly from \eqref{identities} after noticing that since $z(a)$ and $z(b)$ are central,
\begin{displaymath}
\alpha_{w(b)}(z(a))=z(a) \quad\text{and}\quad \alpha_{z(b)}(w(a))=w(a).
\end{displaymath}
Moreover, for $i<j$ we have $u_iu_j=v_{ij}u_ju_i$.
By \eqref{identities} and the previous lemma, one calculates
\begin{displaymath}
\begin{split}
g(u_n,u_iu_j)&=g(\alpha_{u_j}(u_n),u_i)g(u_n,u_j)\\
            &=g(u_nv_{jn},u_i)g(u_n,u_j)\\
            &=g(u_n,u_i)g(v_{jn},u_i)g(u_n,u_j)
\end{split}
\end{displaymath}
and
\begin{displaymath}
\begin{split}
g(u_n,v_{ij}u_ju_i)&=g(u_n,v_{ij})g(u_n,u_ju_i)\\
                 &=g(u_n,v_{ij})g(\alpha_{u_i}(u_n),u_j)g(u_n,u_i)\\
                 &=g(u_n,v_{ij})g(u_nv_{in},u_j)g(u_n,u_i)\\
                 &=g(u_n,v_{ij})g(u_n,u_j)g(v_{in},u_j)g(u_n,u_i),
\end{split}
\end{displaymath}
so that
\begin{equation}\label{parameter-relation}
g(v_{jn},u_i)=g(u_n,v_{ij})g(v_{in},u_j),
\end{equation}
which gives the last identity in the second line of the statement.
Finally, we compute
\begin{displaymath}
\begin{split}
g(a,bb')&=g(\alpha_{b'}(a),b)g(a,b')=g(a+(0,b'_1a_n,\dotsc,b'_{n-1}a_n),b)g(a,b')\\
&=g((0,b'_1a_n,\dotsc,b'_{n-1}a_n),w(b))g(a,b)g(a,b')\\
&=\Big(\prod_{i=1}^{n-1} g(v_{in},w(b))^{b'_ia_n}\Big)g(a,b)g(a,b')\\
&=\Big(\prod_{i=1}^{n-1}\Big(\prod_{j=1}^{n-1}g(v_{in},u_j)^{b_j}\Big)^{b'_ia_n}\Big)g(a,b)g(a,b').
\end{split}
\end{displaymath}
\end{proof}

\begin{lems}\label{w-part}
For all $a\in H(n)$ and $b\in G(n-1)$ we have
\begin{displaymath}
\begin{split}
g(w(a),w(b))&=\Big(\prod_{i=1}^{n-1}\lambda_i^{\frac{1}{2}b_ia_n(a_n-1)}g(v_{in},u_i)^{\frac{1}{2}a_nb_i(b_i-1)}\Big)\\
&\cdot\prod_{1\leq i<j\leq n-1}g(v_{in},u_j)^{b_ib_ja_n}.
\end{split}
\end{displaymath}
\end{lems}

\begin{proof}
First we see from \eqref{b-homomorphism} that if $b_j\geq 1$, then
\begin{displaymath}
\begin{split}
g(u_n,u_j^{b_j})&=g(u_n,u_j^{b_j-1}u_j)\\
  &=g(v_{jn},u_j)^{b_j-1}g(u_n,u_j^{b_j-1})g(u_n,u_j)\\
  &=\dotsm=g(v_{jn},u_j)^{\frac{1}{2}b_j(b_j-1)}g(u_n,u_j)^{b_j}
\end{split}
\end{displaymath}
and then it is not hard to see that
\begin{displaymath}
g(u_n,u_j^{b_j})=g(v_{jn},u_j)^{\frac{1}{2}b_j(b_j-1)}g(u_n,u_j)^{b_j}
\end{displaymath}
for negative $b_j$ as well, for example by applying \eqref{b-homomorphism} again.

Moreover, note that $w(b)=u_{n-1}^{b_{n-1}}\dotsm u_1^{b_1}$, so that by \eqref{b-homomorphism},
\begin{displaymath}
\begin{split}
g(u_n,w(b))&=g(u_n,u_{n-1}^{b_{n-1}}\dotsm u_1^{b_1})\\
&=\big(\prod_{j=2}^{n-1}g(v_{1n},u_j)^{b_1b_j}\big)g(u_n,u_{n-1}^{b_{n-1}}\dotsm u_2^{b_2})g(u_n,u_1^{b_1})\\
&=\dotsm=\big(\prod_{1\leq i<j\leq n-1}g(v_{in},u_j)^{b_ib_j}\big)\big(\prod_{j=1}^{n-1}g(u_n,u_j^{b_j})\big).
\end{split}
\end{displaymath}
Then by \eqref{g-Hn} for $a_n\geq 1$,
\begin{displaymath}
\begin{split}
g(w(a),w(b))&=g(a_nu_n,u_{n-1}^{b_{n-1}}\dotsm u_1^{b_1})\\
          &=\Big(\prod_{i=1}^{n-1}\lambda_i^{b_i(a_n-1)}\Big)\cdot g((a_n-1)u_n,u_{n-1}^{b_{n-1}}\dotsm u_1^{b_1})g(u_n,u_{n-1}^{b_{n-1}}\dotsm u_1^{b_1})\\
          &=\dotsm=\Big(\prod_{i=1}^{n-1}\lambda_i^{b_i\cdot\frac{1}{2}a_n(a_n-1)}\Big)\cdot g(u_n,u_{n-1}^{b_{n-1}}\dotsm u_1^{b_1})^{a_n}\\
          &=\Big(\prod_{i=1}^{n-1}\lambda_i^{\frac{1}{2}b_ia_n(a_n-1)}\Big)\cdot\Big(\prod_{1\leq i<j\leq n-1}g(v_{in},u_j)^{b_ib_ja_n}\Big)\\
          &\cdot\Big(\prod_{j=1}^{n-1}g(v_{jn},u_j)^{\frac{1}{2}a_nb_j(b_j-1)}g(u_n,u_j)^{a_nb_j}\Big).
\end{split}
\end{displaymath}
Again, it is not hard to see that a similar argument also works for negative $a_n$.
Finally, recall that we have chosen $g$ so that $g(u_n,u_j)=1$ by \eqref{g-trivial}.
\end{proof}

\begin{lems}
We have
\begin{displaymath}
\widetilde{H}^2(G(n),\mathbb{T})\cong\mathbb{T}^{n(n-1)},
\end{displaymath}
and for each set of $n(n-1)$ parameters
\begin{displaymath}
\{\lambda_{i,jn} : 1\leq i\leq n, 1\leq j\leq n-1\}\subseteq\mathbb{T},
\end{displaymath}
the associated $[\tau]$ in $\widetilde{H}^2(G(n),\mathbb{T})$ may be represented by
\begin{displaymath}
\begin{split}
\tau(a'b,ab')&=\prod_{1\leq i<j\leq n-1}\lambda_{i,jn}^{a_{jn}b_i+a_nb_{ij}}\lambda_{j,in}^{a_{in}b_j+a_n(b_ib_j-b_{ij})}\prod_{j=1}^{n-1}\lambda_{j,jn}^{a_{jn}b_j+\frac{1}{2}a_nb_j(b_j-1)}\\
             &\quad\cdot\prod_{j=1}^{n-1}\lambda_{n,jn}^{a'_n(a_{jn}+b_ja_n)+\frac{1}{2}b_ja_n(a_n-1)}.
\end{split}
\end{displaymath}
\end{lems}

\begin{proof}
If one puts $\lambda_{i,jn}=g(v_{jn},u_i)$ for $i,j<n$ and $\lambda_{n,jn}=\lambda_j$ for $j<n$, then this is a consequence of the preceding lemmas.
Indeed, by \eqref{tilde-formula} we can represent $\tau$ as a pair $(\sigma_{H(n)},g)$.
Here $\sigma_{H(n)}$ is of the form \eqref{sigma-Hn} and $g$ can decomposed as in Lemma~\ref{decomposition}
with factors computed in Lemma~\ref{z-part} and Lemma~\ref{w-part}.
\end{proof}

To complete the proof of Theorem~\ref{main-theorem}, we set $r=a'b$ and $s=ab'$
and recall that by Corollary~\ref{inductively} we can compute $\sigma_n$ inductively as $[\sigma_n]=\prod_{k=2}^n[\tau_n]$.

Finally, we can also check that $\sum_{k=2}^n k(k-1)=\frac{1}{3}(n+1)n(n-1)$.

\section{The twisted group \texorpdfstring{$C^*$}{C*}-algebras \texorpdfstring{$C^*(G(n),\sigma)$}{C*(G(n),sigma)} of \texorpdfstring{$G(n)$}{G(n)}}\label{algebras}

Again, let $G$ be \emph{any} discrete group, $\sigma$ a two-cocycle of $G$ and $\mathcal{H}$ a nontrivial Hilbert space.
A map $U$ from $G$ into the unitary group of $\mathcal{H}$ satisfying
\begin{displaymath}
U(r)U(s)=\sigma(r,s)U(rs)
\end{displaymath}
for all $r,s\in G$ is called a \emph{$\sigma$-projective unitary representation of $G$ on $\mathcal{H}$}.

We recall the following facts about twisted group $C^*$-algebras and refer to Zeller-Meier \cite{Zeller-Meier} for further details of the construction.

To each pair $(G,\sigma)$, we may associate the full twisted group $C^*$-algebra $C^*(G,\sigma)$.
Denote the canonical injection of $G$ into $C^*(G,\sigma)$ by $i_{\sigma}$.
Then $C^*(G,\sigma)$ satisfies the following universal property.
Every $\sigma$-projective unitary representation of $G$ on some Hilbert space $\mathcal{H}$ (or in some unital $C^*$-algebra $A$)
factors uniquely through $i_{\sigma}$.

The reduced twisted group $C^*$-algebra $C^*_r(G,\sigma)$ is generated by
the left regular $\sigma$-projective unitary representation $\lambda_{\sigma}$ of $G$ on $B(\ell^2(G))$.
Consequently, $\lambda_{\sigma}$ extends to a $^*$-homomorphism of $C^*(G,\sigma)$ onto $C^*_r(G,\sigma)$.
If $G$ is amenable, then $\lambda_{\sigma}$ is faithful.
Note especially that every nilpotent group is amenable,
so that $C^*(G(n),\sigma)\cong C^*_r(G(n),\sigma)$ through $\lambda_{\sigma}$ for every $n\geq 1$ and all two-cocycles $\sigma$ of $G(n)$.

Finally, we remark that if $\tau\sim\sigma$ through some function $\beta\colon G\to\mathbb{T}$,
then the assignment $i_{\tau}(r)\mapsto\beta(r)i_{\sigma}(r)$ induces an isomorphism $C^*(G,\tau)\to C^*(G,\sigma)$.

\begin{theorem}[Remark~3.1 in the published version]\label{universal}
Fix $n\geq 2$, let $\sigma$ be a two-cocycle of $G(n)$ of the form \eqref{main-formula}, that is, determined by the $\frac{1}{3}(n+1)n(n-1)$ parameters
\begin{displaymath}
\{\lambda_{i,jk} : 1\leq i\leq k, 1\leq j<k\leq n\}\subseteq\mathbb{T},
\end{displaymath}
and set
\begin{equation}\label{lambda-dependence}
\lambda_{k,ij}=\overline{\lambda_{i,jk}}\lambda_{j,ik}
\end{equation}
when ${1\leq i<j<k\leq n}$.

Then the twisted group $C^*$-algebra $C^*(G(n),\sigma)$
is the universal $C^*$-algebra generated by unitaries $\{U_i\}_{1\leq i\leq n}$ and $\{V_{jk}\}_{1\leq j<k\leq n}$ satisfying the relations
\begin{equation}\label{relations}
[V_{jk},V_{lm}]=I, \quad [U_i,V_{jk}]=\lambda_{i,jk}I, \quad [U_j,U_k]=V_{jk}
\end{equation}
for ${1\leq i\leq n}$, ${1\leq j<k\leq n}$, and ${1\leq l<m\leq n}$.
\end{theorem}

\begin{proof}
Set $U_i=i_{\sigma}(u_i)$ and $V_{jk}=i_{\sigma}(v_{jk})$ and note that \eqref{main-formula} gives that $\sigma(u_i,v_{jk})=\lambda_{i,jk}$
and $\sigma(v_{jk},u_i)=1$ for all ${1\leq i\leq n}$ and ${1\leq j<k\leq n}$.
Thus,
\begin{displaymath}
[U_i,V_{jk}]=\sigma(u_i,v_{jk})\overline{\sigma(v_{jk},u_i)}I=\lambda_{i,jk}I\,\text{ for all }1\leq i\leq n,\,1\leq j<k\leq n.
\end{displaymath}
Moreover, note that $\sigma(u_i,u_j)=1$ for all ${1\leq i,j\leq n}$ and $\sigma(v_{jk},v_{lm})=1$ for all ${1\leq j<k\leq n}$ and ${1\leq l<m\leq n}$.
Hence, it is clear that $C^*(G(n),\sigma)$ is generated as a $C^*$-algebra by unitaries satisfying \eqref{relations}.

Next, suppose that $A$ is any $C^*$-algebra generated by a set of unitaries satisfying the relations \eqref{relations}.
For each $r$ in $G(n)$ we define the unitary $W_r$ in $A$ by
\begin{displaymath}
W_r=V_{12}^{r_{12}}\cdots V_{n-1,n}^{r_{n-1,n}}\cdot U_n^{r_n}\cdots U_1^{r_1}.
\end{displaymath}
Then a computation using \eqref{relations} repeatedly gives that\footnote{In general,
it will require much work to compute the formula for $\tau$ and it is not needed for this argument.
However, for $n=2$, the expression for $\tau$ is precisely of the form \eqref{G(2)-formula}.} $W_rW_s=\tau(r,s)W_{rs}$,
where $\tau(r,s)$ is a scalar in $\mathbb{T}$ for all $r,s\in G(n)$.
Now, the associativity of $A$ immediately implies that $\tau$ is a two-cocycle of $G(n)$,
so that $W$ is a $\tau$-projective unitary representation of $G(n)$ in $A$. 
Furthermore, note that $\tau$ satisfies
\begin{displaymath}
\tau(u_i,v_{jk})\overline{\tau(v_{jk},u_i)}=\lambda_{i,jk}\,\text{ for }1\leq i\leq n,\, 1\leq j<k\leq n.
\end{displaymath}
By the universal property of the full twisted group $C^*$-algebra,
there exists a unique $^*$-homomorphism $\varphi$ of $C^*(G(n),\tau)$ onto $A$ such that $\varphi(i_{\tau}(r))=W(r)$ for all $r\in G(n)$.

Therefore, it is sufficient to show that $\tau\sim\sigma$, because then $C^*(G(n),\tau)$ is canonically isomorphic to $C^*(G(n),\sigma)$.
By Theorem~\ref{main-theorem}, there is some $\beta\colon G(n)\to\mathbb{T}$ such that $\sigma'$,
given by
\begin{displaymath}
\sigma'(r,s)=\beta(r)\beta(s)\overline{\beta(rs)}\tau(r,s),
\end{displaymath}
is of the form \eqref{main-formula}.
We calculate that
\begin{displaymath}
\begin{split}
&\sigma'(u_i,v_{jk})\overline{\sigma'(v_{jk},u_i)}\\
&=\beta(u_i)\beta(v_{jk})\overline{\beta(u_iv_{jk})}\tau(u_i,v_{jk})\overline{\beta(v_{jk})}\overline{\beta(u_i)}\beta(v_{jk}u_i)\overline{\tau(v_{jk},u_i)}\\
&=\tau(u_i,v_{jk})\overline{\tau(v_{jk},u_i)}=\lambda_{i,jk}
\end{split}
\end{displaymath}
for all ${1\leq i\leq n}$ and ${1\leq j<k\leq n}$.
Hence, $\sigma'=\sigma$, so $\tau\sim\sigma$.
\end{proof}

The above relation \eqref{lambda-dependence} is a consequence of \eqref{parameter-relation} in the proof of Theorem~\ref{main-theorem}
and is the reason why $\lambda_{i,jk}$ for $i>k$ is not involved in the expression \eqref{main-formula}.
To illustrate this further, consider the three-dimensional case.
Let $U_1,U_2,U_3$ and $V_{12},V_{13},V_{23}$ be unitaries in a $C^*$-algebra $B$ satisfying
\begin{displaymath}
[V_{jk},V_{lm}]=I, \quad [U_i,V_{jk}]=\mu_{i,jk}I, \quad [U_j,U_k]=V_{jk}
\end{displaymath}
for ${1\leq i\leq 3}$, ${1\leq j<k\leq 3}$, and ${1\leq l<m\leq 3}$ where $\{\mu_{i,jk}\}$ is \emph{any} set of nine scalars in $\mathbb{T}$.
Then we can compute that
\begin{displaymath}
\begin{split}
U_1U_2U_3&=V_{12}U_2U_1U_3=\dotsm =\mu_{2,13}V_{12}V_{13}V_{23}U_3U_2U_1,\\
U_1U_2U_3&=U_1V_{23}U_3U_2=\dotsm =\mu_{1,23}\mu_{3,12}V_{12}V_{13}V_{23}U_3U_2U_1,
\end{split}
\end{displaymath}
that is, we must have $\mu_{2,13}=\mu_{1,23}\mu_{3,12}$.

For dimensions $n>3$, any choice of a triple of unitaries from the family $\{U\}_{i=1}^n$ gives a similar dependence.
In the $n\cdot\frac{1}{2}n(n-1)$ commutation relations, these $\binom{n}{3}$ dependencies are the only possible ones since
\begin{displaymath}
\textstyle{n\cdot\frac{1}{2}n(n-1)-\binom{n}{3}=\frac{1}{2}n(n-1)\left(n-\frac{1}{3}(n-2)\right)=\frac{1}{3}(n+1)n(n-1).}
\end{displaymath}

\begin{remark}\label{conjecture}
For $n\geq 2$, let $\omega$ be the dual two-cocycle of $G(n)$, that is,
\begin{displaymath}
\omega\colon G(n)\times G(n)\to\widehat{H^2(G(n),\mathbb{T})}\cong\mathbb{Z}^{\frac{1}{3}(n+1)n(n-1)}
\end{displaymath}
is determined by $\omega(r,s)(\sigma)=\sigma(r,s)$ for a two-cocycle $\sigma$ of $G(n)$.
Let the group $R(G(n))$ be defined as the set $\mathbb{Z}^{\frac{1}{3}(n+1)n(n-1)}\times G(n)$ with product
\begin{displaymath}
(j,r)(k,s)=(j+k+\omega(r,s),rs).
\end{displaymath}
It is not entirely obvious that $\omega$ and $R(G(n))$ are well-defined
and we refer to \cite[p.~689--690]{Packer-Raeburn} and \cite[Section~4]{EW} for details on this and the fact that $R(G(n))$ is a representation group for $G(n)$.
Moreover, according to \cite[Corollary~1.3]{Packer-Raeburn} we may construct a continuous field $A$ over $H^2(G(n),\mathbb{T})$
with fibers $A_{\lambda}\cong C^*(G(n),\sigma_{\lambda})$ for each $\lambda\in H^2(G(n),\mathbb{T})$.
Then the $C^*$-algebra associated with this continuous field will be naturally isomorphic to the group $C^*$-algebra of the group $R(G(n))$.

Next, we briefly consider the group $G(3,2)$ generated by $u_1,u_2,v_{12},w_1,w_2$ satisfying
\begin{displaymath}
[u_1,u_2]=v_{12},\quad [u_1,v_{12}]=w_1,\quad [u_2,v_{12}]=w_2,\quad w_1,w_2\text{ central}.
\end{displaymath}
Then we have $Z(G(3,2))\cong\mathbb{Z}^2$ and $Z(C^*(G(3,2)))\cong C(\mathbb{T}^2)$.

The following statement can also be deduced from \cite[Theorem~1.2 and Examples~1.4~(3)]{Packer-Raeburn},
but we include the analysis that follows, because it is similar to that used in Theorem~\ref{continuous-field}.

Let $i$ denote the canonical injection of $G(3,2)$ into $C^*(G(3,2))$.
For each $\lambda=(\lambda_1,\lambda_2)\in\mathbb{T}^2$, let $C^*(G(2),\sigma_{\lambda})$ be generated by unitaries satisfying \eqref{relations}.
By a similar argument as in Theorem~\ref{continuous-field}, there is a surjective $^*$-homomorphism
\begin{displaymath}
\pi_{\lambda}\colon C^*(G(3,2))\to C^*(G(2),\sigma_{\lambda})
\end{displaymath}
such that $i(u_i)=U_i$, $i(v_{12})=V_{12}$, and $i(w_i)=\lambda_iI$ for $i=1,2$.
Moreover, the kernel of $\pi_{\lambda}$ coincides with the ideal of $C^*(G(3,2))$ generated by
\begin{displaymath}
\lambda\in\operatorname{Prim}{Z(C^*(G(3,2)))}\cong\widehat{Z(C^*(G(3,2)))}=\mathbb{T}^2\cong H^2(G(2),\mathbb{T}).
\end{displaymath}

Again, similarly as in Theorem~\ref{continuous-field}, we define a set of sections and apply the Dauns-Hofmann Theorem.
In this way, the triple
\begin{displaymath}
\left(H^2(G(2),\mathbb{T}),\{C^*(G(2),\sigma_{\lambda})\}_{\lambda},\widetilde{C^*(G(3,2))}\right)
\end{displaymath}
is a full continuous field of $C^*$-algebras, and the $C^*$-algebra associated with this continuous field is naturally isomorphic to $C^*(G(3,2))$.

It is not difficult to see that $R(G(2))$ is isomorphic to $G(3,2)$.
We conjecture that $R(G(n))\cong G(3,n)$ also for $n\geq 3$,
where $G(3,n)$ is the free nilpotent group of class $3$ and rank $n$ as described in \eqref{class-3}, so that $A$ is isomorphic to $C^*(G(3,n))$.
For $n\geq 3$, the complicated part is to construct a homomorphism $R(G(n)) \to G(3,n)$,
find an isomorphism $\mathbb{Z}^{\frac{1}{3}(n+1)n(n-1)}\cong Z(G(3,n))$,
and then use \eqref{quotient} to produce a commuting diagram:
\begin{displaymath}
\begin{tikzcd}[column sep=large,row sep=large]
1 \arrow{r} & \mathbb{Z}^{\frac{1}{3}(n+1)n(n-1)} \arrow{r}\arrow{d}[swap]{\cong} & R(G(n)) \arrow{r}\arrow{d} & G(n) \arrow{r}\arrow{d}{=} & 1 \\
1 \arrow{r} & Z(G(3,n)) \arrow{r} & G(3,n) \arrow{r} & G(n) \arrow{r} & 1 
\end{tikzcd}
\end{displaymath}
In fact,
\cite[Proposition~2.2 and Remark~2.3]{Szymik} indicate that the representation group $R(G(m,n))$ for $G(m,n)$ defined similarly as above
may be isomorphic to $G(m+1,n)$ for all $m,n\geq 1$.

\end{remark}

\section{Simplicity of the twisted group \texorpdfstring{$C^*$}{C*}-algebras \texorpdfstring{$C^*(G(n),\sigma)$}{C*(G(n),sigma)}}\label{simplicity}

Let $\sigma$ be a two-cocycle of any group $G$.
An element $r$ of $G$ is called \emph{$\sigma$-regular} if $\sigma(r,s)=\sigma(s,r)$ whenever $s$ in $G$ commutes with $r$.
If $r$ is $\sigma$-regular, then every conjugate of $r$ is also $\sigma$-regular.
Therefore, we say that a conjugacy class of $G$ is $\sigma$-regular if it contains a $\sigma$-regular element.

\vspace{1em}

Let $n\geq 2$.
The conjugacy class $C_r$ of $r\in G(n)$ is infinite if $r\notin V(n)=Z(G(n))$.
Indeed, for any $s\in G(n)$ we have
\begin{equation*}
(srs^{-1})_i=r_i\text{ and }(srs^{-1})_{jk}=r_{jk}+s_jr_k-r_js_k.
\end{equation*}
Hence, $|C_r|=\infty$ if $r_i\neq 0$ for some $i$.
Of course, $C_r=\{r\}$ if $r\in V(n)$.

Now, we fix a two-cocycle $\sigma$ of $G(n)$ of the form \eqref{main-formula}.

\begin{lemma}\label{center}
Let $S(G(n))$ be the set of $\sigma$-regular central elements of $G(n)$, that is,
\begin{displaymath}
S(G(n))=\{r\in V(n) \mid \sigma(r,s)=\sigma(s,r)\text{ for all }s\in G(n)\}.
\end{displaymath}
Then $S(G(n))$ is a subgroup of $G(n)$ and $Z(C^*(G(n)),\sigma)\cong C(\widehat{S(G(n))})$.
\end{lemma}

\begin{proof}
It is not hard to check that $S(G(n))$ is a subgroup of $V(n)$.

We identify $C^*(G(n), \sigma)$ with $C_r^*(G(n),\sigma)\subseteq B(\ell^2(G))$.
Let $\delta_e$ in $\ell^2(G)$ be the characteristic function on $\{e\}$ and for an operator $T$ in $B(\ell^2(G))$, set $f_T=T\delta_e\in\ell^2(G)$.
If $T$ belongs to the center of $C^*(G(n),\sigma)$, then $f_T$ can be nonzero only on the finite $\sigma$-regular conjugacy classes of $G(n)$,
that is, on $S(G(n))$ (see e.g.\ \cite[Lemmas~2.3 and 2.4]{Omland}).

Next, let $C^*(S(G(n)),\sigma)$ be identified with $\{\lambda_\sigma(s) \mid s\in S(G(n))\}\subseteq B(\ell^2(G))$.
This means that $Z(C^*(G(n),\sigma)\subseteq C^*(S(G(n)),\sigma)$.
As the reverse inclusion obviously holds, we have $Z(C^*(G(n)),\sigma)=C^*(S(G(n)),\sigma)$.

Now, it is not difficult to see that $C^*(S(G(n)))\cong C^*(S(G(n)),\sigma)$.
Indeed, as $s\mapsto\lambda_{\sigma}(s)$ is a unitary representation of $S(G(n))$ into $C^*(S(G(n)),\sigma)$ and
the canonical tracial state $\tau$ on $C^*(S(G(n)),\sigma)$ is faithful and satisfies $\tau(\lambda_\sigma(s))=0$ for each nonzero $s\in S(G(n))$,
this is just a consequence of \cite[Th\'eor\`eme 4.22]{Zeller-Meier}. 
Altogether, we get
\begin{displaymath}
Z(C^*(G(n)),\sigma)=C^*(S(G(n)),\sigma)\cong C^*(S(G(n)))\cong C(\widehat{S(G(n))}).\qedhere
\end{displaymath}
\end{proof}

\begin{remark}
If $S(G(n))$ is nontrivial, we can describe $C^*(G(n),\sigma)$ as a continuous field of $C^*$-algebras over the base space $\widehat{S(G(n))}$.
The fibers will be isomorphic to $C^*(G(n)/S(G(n)),\omega)$ for some two-cocycle $\omega$ of $G(n)/S(G(n))$
(see \cite[Theorem~1.1]{Lee-Packer-N} and \cite[Theorem~1.2]{Packer-Raeburn} for further details).
\end{remark}

\begin{example}[{\cite[Lemma~3.8 and Theorem~3.9]{Lee-Packer-N}}]
Fix a two-cocycle $\sigma$ of $G(2)$ of the form \eqref{G(2)-formula} such that both $\lambda_{1,12}$ and $\lambda_{2,12}$ are torsion elements.
Let $p$ and $q$ be the smallest natural numbers such that $\lambda_{1,12}^p=\lambda_{2,12}^q=1$ and set $k=\operatorname{lcm}(p,q)$.
Clearly, $V(2)=\mathbb{Z}$ and $S(G(2))=k\mathbb{Z}$.
Moreover, $G(2)/S(G(2))$ can be identified with the group with product
\begin{displaymath}
(r_1,r_2,r_{12})(s_1,s_2,s_{12})=(r_1+s_1,r_2+s_2,r_{12}+s_{12}+r_1s_2\,\operatorname{mod}k\mathbb{Z} )
\end{displaymath}
for $r_1,r_2,s_1,s_2\in\mathbb{Z}$ and $r_{12},s_{12}\in\{0,1,\dotsc,k-1\}$.

Then $C^*(G(2),\sigma)$ is a continuous field of $C^*$-algebras over the base space $\widehat{S(G(2))}\cong\mathbb{T}$.
The fibers will be isomorphic to $C^*(G(n)/S(G(n)),\omega_{\lambda})$, where $\lambda\in\mathbb{T}$ and
\begin{displaymath}
\omega_{\lambda}(r,s)=\sigma(r,s)\mu^{r_1s_2}
\end{displaymath}
for some $\mu\in\mathbb{T}$ with $\mu^k=\lambda$.
\end{example}

Characterizations for simplicity of of twisted group $C^*$-algebras of two-step nilpotent groups
have been given in \cite[Corollary~1.4]{Lee-Packer-N} and \cite[Corollary~1.6]{Packer-Raeburn}.
For the groups $G(n)$, the necessary and sufficient conditions for simplicity are somewhat easier to provide.

\begin{theorem}\label{simplicity-theorem}
The following are equivalent:
\begin{enumerate}
\item[(i)] $C^*(G(n),\sigma)$ is simple.
\item[(ii)] $C^*(G(n),\sigma)$ has trivial center.
\item[(iii)] There are no nontrivial central $\sigma$-regular elements in $G(n)$.
\end{enumerate}
\end{theorem}

\begin{proof}
By \cite[Theorem~1.7]{Packer-N} $C^*(G(n),\sigma)$ is simple if and only if every nontrivial $\sigma$-regular conjugacy class of $G(n)$ is infinite.
Since every finite conjugacy class of $G(n)$ is a one-point set of a central element, then (i) is equivalent with (iii).

Moreover, (iii) is the same as saying that $S(G(n))$ is trivial, so therefore, (ii) is equivalent with (iii) by Lemma~\ref{center}.
This also follows from \cite[Theorem~2.7]{Omland}.
\end{proof}

\begin{lemma}\label{sigma-regular}
A central element $s=(0,\dotsc,0,s_{12},s_{13},\dotsc,s_{n-1,n})$ of $G(n)$ is $\sigma$-regular if and only if
\begin{displaymath}
\prod_{1\leq j<k\leq n} \lambda_{i;jk}^{s_{jk}}=1
\end{displaymath}
for all $1\leq i\leq n$.
\end{lemma}

\begin{proof}
Clearly, an element $s=(0,\dotsc,0,s_{12},s_{13},\dotsc,s_{n-1,n})\in V(n)$ is $\sigma$-regular
if and only if $\sigma(s,r)=\sigma(r,s)$ for all $r\in G(n)$.
By a direct calculation from the cocycle formula \eqref{main-formula}, we get that 
\begin{displaymath}
\begin{split}
\sigma(r,s)\overline{\sigma(s,r)}&=
\Big(\prod_{i<j<k}\lambda_{i,jk}^{s_{jk}r_i}\lambda_{j,ik}^{s_{ik}r_j}\Big)
\Big(\prod_{j<k}\lambda_{j,jk}^{s_{jk}r_j}\lambda_{k,jk}^{r_ks_{jk}}\Big)
\Big(\prod_{i<j<k}\lambda_{i,jk}^{-r_ks_{ij}}\lambda_{j,ik}^{r_ks_{ij}}\Big)\\
&=\prod_{i=1}^n\Big(\prod_{1\leq j<k\leq n} \lambda_{i,jk}^{s_{jk}}\Big)^{r_i}
\end{split}
\end{displaymath}
is equal to $1$ for all $r\in G(n)$ if and only if the inner parenthesis is $1$ for each $1\leq i\leq n$.
\end{proof}

\begin{corollary}\label{simplicity-formula}
We have that $C^*(G(n),\sigma)$ is simple if and only if
for each nontrivial central element $s=(0,\dotsc,0,s_{12},s_{13},\dotsc,s_{n-1,n})$ there is some $1\leq i\leq n$ such that
\begin{displaymath}
\prod_{1\leq j<k\leq n} \lambda_{i,jk}^{s_{jk}}\neq 1.
\end{displaymath}
\end{corollary}

\begin{example}
In particular, $C^*(G(3),\sigma)$ is simple if and only if for each nontrivial central element $s=(0,0,0,s_{12},s_{13},s_{23})$ at least one of the following hold:
\begin{displaymath}
\begin{split}
\lambda_{1,12}^{s_{12}}\lambda_{1,13}^{s_{13}}\lambda_{1,23}^{s_{23}}&\neq 1,\\
\lambda_{2,12}^{s_{12}}\lambda_{2,13}^{s_{13}}\lambda_{2,23}^{s_{23}}&\neq 1,\\
\lambda_{3,12}^{s_{12}}\lambda_{3,13}^{s_{13}}\lambda_{3,23}^{s_{23}}&\neq 1.
\end{split}
\end{displaymath}
\end{example}

Next, set $\lambda_{i,jk}=e^{2\pi it_{i,jk}}$ for $t_{i,jk}\in [0,1)$
and consider the $n\times\frac{1}{2}n(n-1)$-matrix $T$ with entries $t_{i,jk}$ in the corresponding spots.
Then $T$ induces a linear map
\begin{displaymath}
\mathbb{R}^{\frac{1}{2}n(n-1)}\to\mathbb{R}^n.
\end{displaymath}

\begin{corollary}\label{matrix-T}
Let $T$ be the matrix described above.
Then following are equivalent:
\begin{enumerate}
\item[(i)] $C^*(G(n),\sigma)$ is simple
\item[(ii)] $T^{-1}(\mathbb{Z}^n)\cap\mathbb{Z}^{\frac{1}{2}n(n-1)} = \{0\}$
\item[(iii)] $T(\mathbb{Z}^{\frac{1}{2}n(n-1)}\setminus\{0\})\cap\mathbb{Z}^n =\varnothing$
\end{enumerate}
\end{corollary}

\begin{remark}
Clearly, condition (ii) above is equivalent to that $T$ restricts to an injective map
\begin{displaymath}
\mathbb{Z}^{\frac{1}{2}n(n-1)}\to\mathbb{R}^n/\mathbb{Z}^n\cong\mathbb{T}^n.
\end{displaymath}
\end{remark}

Furthermore, for $1\leq j<k\leq n$, define
\begin{displaymath}
\Lambda_{jk}=\{t_{i,jk}\in [0,1),\, 1\leq i\leq n \mid e^{2\pi i t_{i,jk}}=\lambda_{i,jk}\}
\end{displaymath}
and for $1\leq i\leq n$, define
\begin{displaymath}
\Lambda_i=\{t_{i,jk}\in [0,1),\, 1\leq j<k\leq n \mid e^{2\pi i t_{i,jk}}=\lambda_{i,jk}\}.
\end{displaymath}

\begin{proposition}
If there exists $i$ such that all the elements of $\Lambda_i$ are irrational and linearly independent over $\mathbb{Q}$,
then $C^*(G(n),\sigma)$ is simple.
\end{proposition}

\begin{proof}
It follows immediately from Lemma~\ref{sigma-regular}, that ``equation $i$'' cannot be satisfied unless $s=0$.
Hence, no nontrivial $\sigma$-regular central elements exists.
\end{proof}

\begin{proposition}
If there exists $j<k$ such that $\Lambda_{jk}$ consists of only rational elements, then $C^*(G(n),\sigma)$ is not simple.
\end{proposition}

\begin{proof}
Let $q$ be the least common two-cocycle of the denominators of the elements of $\Lambda_{jk}$.
Then $qv_{jk}$ is central and $\sigma$-regular. Indeed,
\begin{displaymath}
\sigma(r,qv_{jk})\overline{\sigma(qv_{jk},r)}=\prod_{i=1}^{n-1}\lambda_{i,jk}^{qr_i}=1
\end{displaymath}
for all $r\in G(n)$.
\end{proof}

\begin{remark}
In the case where $C^*(G(n),\sigma)$ is not simple, some more information about the primitive ideal space can be deduced from \cite[Proposition~1.3]{Lee-Packer-N}.
\end{remark}

\section{On isomorphisms invariants of \texorpdfstring{$C^*(G(n),\sigma)$}{C*(G(n),sigma)}}\label{isomorphisms}

Fix $n\geq 2$ and let $\sigma$ be a two-cocycle of $G(n)$.
If $\varphi$ is an automorphism of $G(n)$, define the two-cocycle $\sigma_{\varphi}$ of $G(n)$ by
\begin{equation}\label{aut-action}
\sigma_{\varphi}(r,s)=\sigma(\varphi(r),\varphi(s)).
\end{equation}
Then it is well-known that the associated twisted group $C^*$-algebras $C^*(G(n),\sigma)$ and $C^*(G(n),\sigma_{\varphi})$ are isomorphic.
Indeed, the map
\begin{displaymath}
i_{(G,\sigma)}(r)\mapsto i_{(G,\sigma_{\varphi})}(\varphi^{-1}(r))
\end{displaymath}
extends to an isomorphism $C^*(G(n),\sigma)\to C^*(G(n),\sigma_{\varphi})$.
Moreover, for any automorphism $\varphi$ of $G(n)$, it is easily seen that two two-cocycles $\sigma$ and $\tau$ of $G(n)$ are similar
if and only if $\sigma_{\varphi}$ and $\tau_{\varphi}$ are similar.
Hence, there is a well-defined group action of the automorphism group $\operatorname{Aut}G(n)$ on $H^2(G(n),\mathbb{T})$
defined by $\varphi\cdot [\sigma]=[\sigma_{\varphi}]$.

Therefore, we will now briefly investigate $\operatorname{Aut}G(n)$.
Let $V(n)^n$ be the subgroup of $\operatorname{Aut}G(n)$
consisting of the automorphisms $G(n)\to G(n)$ of the form $u_i\mapsto z_iu_i$ for $1\leq i\leq n$ and elements $z_i\in V(n)=Z(G(n))$.
In particular, these automorphisms leave all the $v_{jk}$'s fixed, i.e.\ $V(n)^n$ is the subgroup of $\operatorname{Aut}G(n)$ leaving $V(n)$ fixed.
Clearly, $V(n)^n$ contains $\operatorname{Inn}G(n)$.
In fact, in the case $n=2$, we have $V(2)^2=\operatorname{Inn}G(2)$.

\begin{proposition}
There is a split short exact sequence:
\begin{displaymath}
\begin{tikzcd}
1 \arrow{r} & V(n)^n \arrow{r} & \operatorname{Aut}{G(n)} \arrow{r} & \operatorname{GL}(n,\mathbb{Z}) \arrow{r} & 1 
\end{tikzcd}
\end{displaymath}
\end{proposition}

\begin{proof}
Assume that $\varphi$ is any endomorphism $G(n)\to G(n)$.
The image of a central element under $\varphi$ must be central, so $\varphi$ restricts to an endomorphism $\varphi_1\colon V(n)\to V(n)$.
Therefore, $\varphi$ also induces an endomorphism $\varphi_2\colon G(n)/V(n)\to G(n)/V(n)$ determined by $\varphi_2(q(r))=q(\varphi(r))$.
Consider now the following commutative diagram:
\begin{displaymath}
\begin{tikzcd}[column sep=large,row sep=large]
1 \arrow{r} & V(n) \arrow{r}{i}\arrow{d}[swap]{\varphi_1} & G(n) \arrow{r}{q}\arrow{d}[swap]{\varphi} & \mathbb{Z}^n \arrow{r}\arrow{d}{\varphi_2} & 1 \\
1 \arrow{r} & V(n) \arrow{r}{i} & G(n) \arrow{r}{q} & \mathbb{Z}^n \arrow{r} & 1 
\end{tikzcd}
\end{displaymath}
Assume that $\varphi_2$ is an automorphism.
Since $\varphi_2$ is surjective, then for all $u_i$ there is some $s_i\in G(n)$ such that $\varphi(s_i)=z_iu_i$ for some $z_i\in V(n)$.
Hence, for all $j<k$, we have $\varphi_1(s_js_ks_j^{-1}s_k^{-1})=v_{jk}$ and therefore, $\varphi_1$ is surjective.
Every surjective endomorphism of $\mathbb{Z}^n$ is also injective, so $\varphi_1$ is an automorphism as well.
Thus, by the ``short five lemma'', $\varphi$ is an automorphism.

The converse obviously holds and hence, $\varphi$ is an automorphism if and only if $\varphi_2$ is an automorphism.

Furthermore, the construction of $G(n)$ in terms of generators and relations
means that every endomorphism $G(n)\to G(n)$ is uniquely determined by its values at $\{u_i\}_{i=1}^n$.
In particular, we let $\varphi\colon G(n)\to G(n)$ be determined by the pair of matrices given by its entries
\begin{displaymath}
(\varphi(u_i)_j),(\varphi(u_i)_{jk})\in M_n(\mathbb{Z})\times M_{n,\frac{1}{2}n(n-1)}(\mathbb{Z})
\end{displaymath}
so that the induced endomorphism $\varphi_2$ is coming from a matrix in $M_n(\mathbb{Z})$.

By the above argument, the map between endomorphism groups defined by
\begin{equation}\label{aut-map}
\operatorname{End}G(n)\to\operatorname{End}\mathbb{Z}^n,\quad(\varphi(u_i)_j),(\varphi(u_i)_{jk})\mapsto (\varphi(u_i)_j)
\end{equation}
restricts to a surjective map $\operatorname{Aut}G(n)\to\operatorname{Aut}\mathbb{Z}^n=\operatorname{GL}(n,\mathbb{Z})$.

Before concluding the argument, we need the following.
\begin{lemma}[\MakeLowercase{nested inside the proof}]
If $\varphi$ and $\varphi'$ are two endomorphisms of $G(n)$, then
\begin{displaymath}
(\varphi\circ\varphi')(u_i)_j=\sum_{k=1}^n\varphi'(u_i)_k\varphi(u_k)_j.
\end{displaymath}
If $\varphi$ and $\varphi'$ are two endomorphisms of $G(n)$ that both induce the trivial map on $G(n)/V(n)$, then
\begin{displaymath}
(\varphi\circ\varphi')(u_i)_{jk}=\varphi'(u_i)_{jk}+\varphi(u_i)_{jk}.
\end{displaymath}
\end{lemma}

\begin{proof}
For the moment, set $\varphi(u_i)_j=r_{ij}$ and $\varphi'(u_i)_j=s_{ij}$.
Then
\begin{displaymath}
(\varphi\circ\varphi')(u_i)=\varphi(u_n^{s_{in}}\dotsm u_1^{s_{i1}}z)=(u_n^{r_{nn}}\dotsm u_1^{r_{n1}})^{s_{in}}\dotsm (u_n^{r_{1n}}\dotsm u_1^{r_{11}})^{s_{i1}}z'
\end{displaymath}
for some elements $z,z'\in V(n)$.
Moreover, we can change the order of the $u_i$'s in the expression just by replacing $z'$ by another central element $z''$ and thus,
\begin{displaymath}
(\varphi\circ\varphi')(u_i)_j=r_{nj}s_{in}+r_{n-1,j}s_{i,n-1}+\dotsm +r_{1j}s_{i1}=\sum_{k=1}^ns_{ik}r_{kj}.
\end{displaymath}
If both $\varphi_2$ and $\varphi'_2$ are trivial,
then $\varphi(u_i)=z_iu_i$ and $\varphi'(u_i)=z_i'u_i$ for all $1\leq i\leq n$ and some elements $z_i,z_i'\in V(n)$.
Hence, $\varphi(v_{jk})=\varphi'(v_{jk})=v_{jk}$ for all $j<k$ and thus,
\begin{displaymath}
(\varphi\circ\varphi')(u_i)=\varphi(z_i'u_i)=z_i'z_iu_i.
\qedhere
\end{displaymath}
\end{proof}
Therefore, \eqref{aut-map} restricts to a surjective \emph{homomorphism} $\operatorname{Aut}G(n)\to\operatorname{GL}(n,\mathbb{Z})$ with kernel isomorphic to
the group $M_{n,\frac{1}{2}n(n-1)}(\mathbb{Z})$ under addition, that is, to $V(n)^n$.

Moreover, if $A$ is a matrix in $\operatorname{GL}(n,\mathbb{Z})$ with entries $a_{ij}$,
then one can define an automorphism $\varphi_A$ of $G(n)$ by $\varphi_A(u_i)_j=a_{ij}$.
Thus, it should be clear that $\operatorname{GL}(n,\mathbb{Z})$ sits inside $\operatorname{Aut}G(n)$ as a subgroup so that the sequence splits.
\end{proof}

\begin{proposition}
If $\varphi$ belongs to $V(n)^n$, then $\sigma_{\varphi}$ is similar to $\sigma$.
Thus, the action of $V(n)^n$ on $H^2(G(n),\mathbb{T})$ given by \eqref{aut-action} is trivial.
\end{proposition}

\begin{proof}
It is not hard to see that
\begin{displaymath}
\sigma(u_i,v_{jk})\overline{\sigma(v_{jk},u_i)}=\sigma_{\varphi}(u_i,v_{jk})\overline{\sigma_{\varphi}(v_{jk},u_i)},
\end{displaymath}
that is,
\begin{displaymath}
[i_{(G,\sigma)}(u_i),i_{(G,\sigma)}(v_{jk})]=[i_{(G,\sigma_{\varphi})}(u_i),i_{(G,\sigma_{\varphi})}(v_{jk})]
\end{displaymath}
for all $1\leq i\leq n$ and $1\leq j<k\leq n$.
It then follows from the universal property of $C^*(G(n),\sigma)$ described in Theorem~\ref{universal} that $\sigma_{\varphi}\sim\sigma$.
\end{proof}

To describe the $\operatorname{GL}(n,\mathbb{Z})$-action on $H^2(G(n),\mathbb{T})$ requires more work.
To any $A$ in $\operatorname{GL}(n,\mathbb{Z})$ we may associate a square matrix $\tilde{A}$ of dimension $\frac{1}{2}n(n-1)$,
with entries coming from the determinant of all $2\times 2$-matrices inside $A$.
More precisely, if $A=(a_{ij})$, $\tilde{A}$ is given by entries $\tilde{a}_{ij,kl}$ for $i<j,k<l$ such that $\tilde{a}_{ij,kl}=a_{ik}a_{jl}-a_{il}a_{jk}$.
Then $A$ acts on the matrix $T$ defined prior to Corollary~\ref{matrix-T} by $A\cdot T=AT\tilde{A}$.
Tedious computations of commutation relations and use of the universal property of $C^*(G(n),\sigma)$ from Theorem~\ref{universal} now lead to the following result.
\begin{proposition}
Let $\sigma$ and $\sigma'$ be two-cocycles of $G(n)$ of the form \eqref{main-formula}
and let $T$ and $T'$ be the associated matrices of Corollary~\ref{matrix-T}.
If there exists a matrix $A$ in $\operatorname{GL}(n,\mathbb{Z})$ such that $A\cdot T=T'$,
then $C^*(G(n),\sigma)$ and $C^*(G(n),\sigma')$ are isomorphic.
\end{proposition}
For $n=2$, it is shown by Packer \cite[Theorem~2.9]{Packer-N} that $C^*(G(2),\sigma)$ and $C^*(G(2),\sigma')$,
where $\sigma$ and $\sigma'$ are of the form \eqref{main-formula},
are isomorphic if and only if there is a $\operatorname{GL}(2,\mathbb{Z})$-matrix $A$ taking $\sigma$ to $\sigma'$.
Note in this case that $\tilde{A}=\operatorname{det}A=\pm 1$.

For $n \geq 3$, it is at the moment not clear whether the $\operatorname{GL}(n,\mathbb{Z})$-action on $H^2(G(n),\mathbb{T})$ described above
is such that the orbits represent different isomorphism classes of twisted group $C^*$-algebras.
Therefore, the problem of determining the isomorphism classes of $C^*(G(n),\sigma)$ remains open for future investigation.

\begin{remark}
Every two-cocycle $\sigma_{\lambda}$ of $G(n)$ is of the form $e^{2\pi i\widetilde{\sigma}_{\lambda}}$ for some two-cocycle $\widetilde{\sigma}_{\lambda}$ on $G(n,\mathbb{R})$.
Moreover, any pair of two-cocycles $\sigma_{\lambda}$ and $\sigma_{\mu}$ of the form described above
are homotopic in the sense of Packer and Raeburn \cite[Section~4]{Packer-Raeburn-2}.
Hence, one may use \cite[Theorem~4.2 and Corollary~4.5]{Packer-Raeburn-2} to deduce that
\begin{displaymath}
K_i(C^*(G(n)),\sigma_{\lambda})\cong K_i(C^*(G(n)),\sigma_{\mu})\cong K_{\textup{top}}^{i+\frac{1}{2}n(n-1)}(G(n,\mathbb{R})/G(n)).
\end{displaymath}
\end{remark}

\begin{remark}[new]\label{new}
The groups $G(n)$, or more generally all the free nilpotent groups $G(m,n)$ of class $m$ and rank $n$ are finitely generated (torsion-free) nilpotent groups. Let $\sigma$ be a two-cocycle of $G(m,n)$ and suppose that there are no nontrivial central $\sigma$-regular elements in $G(m,n)$, that is, $C^*(G(m,n),\sigma)$ is simple with a unique trace by \cite{Packer-H} (see Theorem~\ref{simplicity-theorem} and the subsequent results above for the case of $G(n)$). Then there exists an irreducible representation $\pi_\sigma$ of $G(m+1,n)$ such that $C^*(G(m,n),\sigma)\cong C^*(\pi_\sigma(G(m+1,n)))$ by using Theorem~\ref{continuous-field}, Remark~\ref{conjecture}, and \eqref{quotient}. Thus it follows from \cite{EckMcK} and \cite{EckGil} that $C^*(G(m,n),\sigma)$ is classified by its ordered $K$-theory within the class of simple, unital, nuclear $C^*$-algebras with finite nuclear dimension that satisfy the universal coefficient theorem.
\end{remark}

\bibliographystyle{plain}

\addcontentsline{toc}{section}{References}

\vspace{1em}

Department of Mathematical Sciences, Norwegian University of Science and Technology (NTNU), NO-7491 Trondheim, Norway.


\end{document}